\documentclass[a4paper, 12pt]{article}

\usepackage[sort&compress]{natbib}
\bibpunct{(}{)}{;}{a}{}{,} 

\usepackage{amsthm, amsmath, amssymb, mathrsfs, multirow, url, subfigure}
\usepackage{graphicx} 
\usepackage{ifthen} 
\usepackage{amsfonts}
\usepackage[usenames]{color}
\usepackage{fullpage}

\theoremstyle{plain} 
\newtheorem{theorem}{Theorem}
\newtheorem{corollary}{Corollary}

\theoremstyle{definition}

\theoremstyle{remark} 
\newtheorem{remark}{Remark}
\newtheorem{example}{Example}

\newcommand{\prob}{\mathsf{P}} 
\newcommand{\E}{\mathsf{E}}

\newcommand{\pl}{\mathsf{pl}}
\newcommand{\mpl}{\mathsf{mpl}}

\newcommand{\bin}{{\sf Bin}}

\newcommand{\unif}{{\sf Unif}}
\newcommand{\nm}{{\sf N}}

\newcommand{\RR}{\mathbb{R}}

\newcommand{\YY}{\mathbb{Y}}

\newcommand{\G}{\mathscr{G}}
\newcommand{\Gbar}{\overline{\mathscr{G}}}
\newcommand{\gbar}{\bar{g}}

\newcommand{\nbrack}{N_{[\,]}}

\newcommand{\veta}{\boldsymbol{\veta}}

\newcommand{\Ybar}{\bar Y}%{\overline{Y}}
\newcommand{\ybar}{\bar y}%{\overline{y}}

\renewcommand{\phi}{\varphi} 
\newcommand{\eps}{\varepsilon}

\newcommand{\iid}{\overset{\text{\tiny iid}}{\,\sim\,}}

\title{Plausibility functions and exact frequentist inference}
\author{
Ryan Martin \\
Department of Mathematics, Statistics, and Computer Science \\
University of Illinois at Chicago \\
\url{rgmartin@uic.edu} }
\date{\today}

\begin{document}

\maketitle 

\begin{abstract}
In the frequentist program, inferential methods with exact control on error rates are a primary focus.  The standard approach, however, is to rely on asymptotic approximations, which may not be suitable.  This paper presents a general framework for the construction of exact frequentist procedures based on plausibility functions.  It is shown that the plausibility function-based tests and confidence regions have the desired frequentist properties in finite samples---no large-sample justification needed.  An extension of the proposed method is also given for problems involving nuisance parameters.  Examples demonstrate that the plausibility function-based method is both exact and efficient in a wide variety of problems.

\smallskip

\end{abstract}

\section{Introduction}
\label{S:intro}

In the Neyman--Pearson program, construction of tests or confidence regions having control over frequentist error rates is an important problem.  But, despite its importance, there seems to be no general strategy for constructing exact inferential methods.  When an exact pivotal quantity is not available, the usual strategy is to select some summary statistic and derive a procedure based on the statistic's asymptotic sampling distribution.  First-order methods, such as confidence regions based on asymptotic normality of the maximum likelihood estimator, are known to be inaccurate in certain problems.  Procedures with higher-order accuracy are also available \citep[e.g.,][]{reid2003, brazzale.davison.reid.2007}, but the more challenging calculations required to implement these approximate methods have led to their relatively slow acceptance in applied work.   

In many cases, numerical methods are needed or even preferred.  Arguably the most popular numerical method in this context is the bootstrap \citep[e.g.,][]{efrontibshirani1993, davison.hinkley.1997}.  The bootstrap is beautifully simple and, perhaps because of its simplicity, has made a tremendous impact on statistics; see the special issue of \emph{Statistical Science} (Volume~18, Issue~2, 2003).  However, the theoretical basis for the bootstrap is also asymptotic, so no claims about exactness of the corresponding method(s) can be made.  The bootstrap also cannot be used blindly, for there are cases where the bootstrap fails and the remedies for bootstrap failure are somewhat counterintuitive \citep[e.g.,][]{bickel1997}.  In light of these subtleties, an alternative and generally applicable numerical method with exact frequentist properties might be desirable.  

In this paper, I propose an approach to the construction of exact frequentist procedures.  In particular, in Section~\ref{SS:construction}, I define a set-function that assigns numerical scores to assertions about the parameter of interest.  This function measures the plausibility that the assertion is true, given the observed data.  Details on how the plausibility function can be used for inference are given in Section~\ref{SS:inference}, but the main idea is that the assertion is doubtful, given the observed data, whenever its plausibility is sufficiently small.  In Section~\ref{SS:theory}, sampling distribution properties of the plausibility function are derived, and, from these results, it follows that hypothesis tests or confidence regions based on the plausibility function have guaranteed control on frequentist error rates in finite samples.  A large-sample result is presented in Theorem~\ref{thm:limit} to justify the claimed efficiency of the method.  Evaluation of the plausibility function, and implementation of the proposed methodology, will generally require Monte Carlo methods; see Section~\ref{SS:computation}.   

Plausibility function-based inference is intuitive, easy to implement, and gives exact frequentist results in a wide range of problems.  In fact, the simplicity and generality of the proposed method makes it easily accessible, even to undergraduate statistics students.  Tastes of the proposed method have appeared previously in the literature but the results are scattered and there seems to be no unified presentation.  For example, the use of p-values for hypothesis testing and confidence intervals is certainly not new, nor is the idea of using Monte Carlo methods to approximate critical regions and confidence bounds \citep{harrison2012, garthwaite.buckland.1992, besag.clifford.1989, bolviken.skovlund.1996}.  Also, the relative likelihood version of the plausibility region appears in \citet{spjotvoll1972}, \citet{Feldman.Cousins.1998}, and \citet{zhang.woodroofe.2002}.  Each of these papers has a different focus, so the point that there is a simple, useful, and very general method underlying these developments apparently has yet to be made.  The present paper makes such a point.   

In many problems, the parameter of interest is some lower-dimensional function, or component, or feature, of the full parameter.  In this case, there is an interest parameter and a nuisance parameter, and I propose a marginal plausibility function for the interest parameter in Section~\ref{S:nuisance}.  For models with a certain transformation structure, the exact sampling distribution results of the basic plausibility function can be extended to the marginal inference case.  What can be done for models without this structure is discussed, and several of examples are given that demonstrate the method's efficiency.

Throughout I focus on confidence regions, though hypothesis tests are essentially the same.  From the theory and examples, the general message is that the plausibility function-based method is as good or better than existing methods.  In particular, Section~\ref{S:bootstrap} presents a simple but practically important random effects model, and it is shown that the proposed method provides exact inference, while the standard parametric bootstrap fails.  Some concluding remarks are given in Section~\ref{S:remarks}.

\section{Plausibility functions}
\label{S:plausibility}

\subsection{Construction}
\label{SS:construction}

Let $Y$ be a sample from distribution $\prob_\theta$ on $\YY$, where $\theta$ is an unknown parameter taking values in $\Theta$, a separable space; here $Y$ could be, say, a sample of size $n$ from a product measure $\prob_\theta^n$, but I have suppressed the dependence on $n$ in the notation.  

To start, let $\ell: \YY \times \Theta \to [0,\infty)$ be a loss function, i.e., a function such that small values of $\ell(y,\theta)$ indicate that the model with parameter $\theta$ fits data $y$ reasonably well.  This loss function $\ell(y,\theta)$ could be a sort of residual sum-of-squares or the negative log-likelihood.  Assume that there is a minimizer $\hat\theta=\hat\theta(y)$ of the loss function $\ell(y,\theta)$ for each $y$.  Next define the function 
\begin{equation}
\label{eq:Ty}
T_{y,\theta} = \exp[-\{\ell(y,\theta) - \ell(y,\hat\theta)\}],
\end{equation}
The focus of this paper is the case where the loss function is the negative log-likelihood, so the function $T_{y,\theta}$ in \eqref{eq:Ty} is the relative likelihood 
\begin{equation}
\label{eq:relative.likelihood}
T_{y,\theta} = L_y(\theta) \, / \, L_y(\hat\theta), 
\end{equation}
where $L_y(\theta)$ is the likelihood function, assumed to be bounded, and $\hat\theta$ is a maximum likelihood estimator.  Other choices of $T_{y,\theta}$ are possible (see Remark~\ref{re:other.choices} in Section~\ref{S:remarks}) but the use of likelihood is reasonable since it conveniently summarizes all information in $y$ concerning $\theta$.  Let $F_\theta$ be the distribution function of $T_{Y,\theta}$ when $Y \sim \prob_\theta$, i.e.,
\begin{equation}
\label{eq:cdf}
F_\theta(t) = \prob_\theta(T_{Y,\theta} \leq t), \quad t \in \RR.
\end{equation}
Often $F_\theta(t)$ will be a smooth function of $t$ for each $\theta$, but the discontinuous case is also possible.  To avoid measurability difficulties, I shall assume throughout that $F_\theta(t)$ is a continuous function in $\theta$ for each $t$.  Take a generic $A \subseteq \Theta$, and define the function 
\begin{equation}
\label{eq:plausibility}
\pl_y(A) = \sup_{\theta \in A} F_\theta(T_{y,\theta}).  
\end{equation}
This is called the \emph{plausibility function}, and it acts a lot like a p-value \citep{impval}.  Intuitively, $\pl_y(A)$ measures the plausibility of the claim ``$\theta \in A$'' given observation $Y=y$.  When $A = \{\theta\}$ is a singleton set, I shall write $\pl_y(\theta)$ instead of $\pl_y(\{\theta\})$.

\subsection{Use in statistical inference}
\label{SS:inference}

The plausibility function can be used for a variety of statistical inference problems.  First, consider a hypothesis testing problem, $H_0: \theta \in \Theta_0$ versus $H_1: \theta \not\in \Theta_0$.  Define a plausibility function-based test as follows:
\begin{equation}
\label{eq:plausibility.test}
\text{reject $H_0$ if and only if $\pl_y(\Theta_0) \leq \alpha$.} 
\end{equation}
The intuition is that if $\Theta_0$ is not sufficiently plausible, given $Y=y$, then one should conclude that the true $\theta$ is outside $\Theta_0$.  In Section~\ref{SS:theory}, I show that the plausibility function-based test \eqref{eq:plausibility.test} controls the probability of Type~I error at level $\alpha$.  

The plausibility function $\pl_y(\theta)$ can also be used to construct confidence regions.  This will be my primary focus throughout the paper.  Specifically, for any $\alpha \in (0,1)$, define the $100(1-\alpha)$\% plausibility region
\begin{equation}
\label{eq:plausibility.region}
\Pi_y(\alpha) = \{\theta: \pl_y(\theta) > \alpha\}.
\end{equation}
The intuition is that $\theta$ values which are sufficiently plausible, given $Y=y$, are good guesses for the true parameter value.  The size result for the test \eqref{eq:plausibility.test}, along with the well-known connection between confidence regions and hypothesis tests, shows that the plausibility regions \eqref{eq:plausibility.region} has coverage at the nominal $1-\alpha$ level.

Before we discuss sampling distribution properties of the plausibility function in the next section, we consider two important fixed-$y$ properties.   These properties motivated the ``unified approach'' developed by \citet{Feldman.Cousins.1998} and further studied by \citet{zhang.woodroofe.2002}.   
\begin{itemize}
\item The minimizer $\hat\theta$ of the loss $\ell(y,\theta)$ satisfies $\pl_y(\hat\theta) = 1$, so the plausibility region is never empty.  In particular, in the case where $\ell$ is the negative log-likelihood, and $T_{y,\theta}$ is the relative likelihood \eqref{eq:relative.likelihood}, the maximum likelihood estimator is contained in the plausibility region.
\vspace{-2mm}
\item The plausibility function is defined only on $\Theta$; more precisely, $\pl_y(\theta) \equiv 0$ for any $\theta$ outside $\Theta$.  So, if $\Theta$ involves some non-trivial constraints, then only parameter values that satisfy the constraint can be assigned positive plausibility.  This implies that the plausibility region cannot extend beyond the effective parameter space.  Compare this to the standard ``$\hat\theta \pm \text{something}$'' confidence intervals or those based on asymptotic normality.  
\end{itemize}

One could also ask if the plausibility region is connected or, perhaps, even convex.  Unfortunately, like Bayesian highest posterior density regions, the plausibility regions are, in general, neither convex nor connected.  An example of non-convexity can be seen in Figure~\ref{fig:gamma.plot}.  That connectedness might fail is unexpected.  Figure~\ref{fig:pois.pl} shows the plausibility function based on a single Poisson sample $Y=2$ using the relative likelihood \eqref{eq:relative.likelihood}; the small convex portion around $\theta=1$ shows that disconnected plausibility regions are possible.  A better understanding of the complicated dual way that the plausibility function depends  on $\theta$---through $F_\theta$ and through $T_{y,\theta}$---is needed to properly explain this phenomenon.  Discreteness of the distribution $F_\theta$ also plays a role, as I have not seen this local convexity in cases where $F_\theta$ is continuous.  However, suppose that $\theta \mapsto \ell(y,\theta)$ is convex and $F_\theta \equiv F$ does not depend on $\theta$; see Section~\ref{SS:computation}.  In this case, the plausibility region takes the form $\{\theta: T_{y,\theta} > F^{-1}(\alpha)\}$, so convexity and connectedness hold by quasi-concavity of $\theta \mapsto T_{y,\theta}$.  For $T_{Y,\theta}$ the relative likelihood \eqref{eq:relative.likelihood}, under standard conditions, $-2\log T_{Y,\theta}$ is asymptotically chi-square under $\prob_\theta$ for any fixed $\theta$, so the limiting distribution function of $T_{Y,\theta}$ is, indeed, free of $\theta$.  Therefore, one could expect convexity of the plausibility region when the sample size is sufficiently large; see Figure~\ref{fig:probit.plot}. 

\begin{figure}
\begin{center}
\scalebox{0.55}{\includegraphics{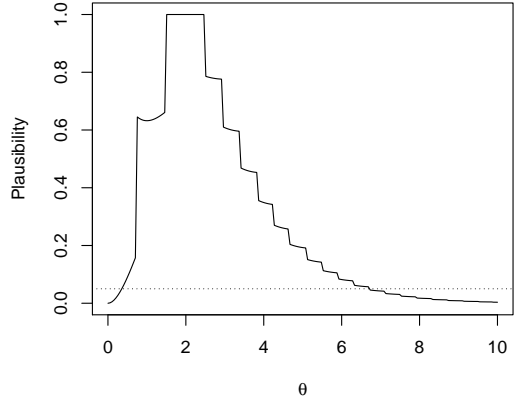}}
%\subfigure[Wide range of $\theta$'s]{\scalebox{0.6}{\includegraphics{pois_pl}}}
%\subfigure[Zoom-in near $\theta=2.4$]{\scalebox{0.6}{\includegraphics{pois_pl_zoom}}}
\end{center}
\caption{Plausibility function for $\theta$ based on a single Poisson sample $Y=2$.}
\label{fig:pois.pl}
\end{figure}

Besides as a tool for constructing frequentist procedures, plausibility functions have some potentially deeper applications.  Indeed, this plausibility function approach has some connections with the new inferential model framework \citep{imbasics, imcond} which employs random sets for valid posterior probabilistic inference without priors.  See Remark~\ref{re:ds.theory} in Section~\ref{S:remarks} for further discussion.  Also, the connection between plausibility functions and p-values is discussed in \citet{impval}.

\subsection{Sampling distribution properties}
\label{SS:theory}

Here I describe the sampling distribution of $\pl_Y(A)$ as a function of $Y \sim \prob_\theta$ for a fixed $A$.  This is critical to the advertised exactness of the proposed procedures.  One technical point: continuity of $F_\theta$ in $\theta$ and separability of $\Theta$ ensure that $\pl_y(A)$ is a measurable function in $y$ for each $A$, so the following probability statements make sense.  

\begin{theorem}
\label{thm:validity}
Let $A$ be a subset of $\Theta$.  For any $\theta \in A$, if $Y \sim \prob_\theta$, then $\pl_Y(A)$ is stochastically larger than uniform.  That is, for any $\alpha \in (0,1)$, 
\begin{equation}
\label{eq:validity}
\sup_{\theta \in A} \prob_\theta\{\pl_Y(A) \leq \alpha\} \leq \alpha.
\end{equation}
\end{theorem}

\begin{proof}%[Proof of Theorem~\ref{thm:validity}]
Take any $\alpha \in (0,1)$ and any $\theta \in A$.  Then, by definition of $\pl_y(A)$ and monotonicity of the probability measure $\prob_\theta$, I get
\begin{equation}
\label{eq:bound}
\prob_\theta\{\pl_Y(A) \leq \alpha\} = \prob_\theta\{\textstyle\sup_{\theta \in A} F_\theta(T_{Y,\theta}) \leq \alpha\} \leq \prob_\theta\{F_\theta(T_{Y,\theta}) \leq \alpha\}. 
\end{equation}
The random variable $T_{Y,\theta}$, as a function $Y \sim \prob_\theta$, may be continuous or not.  In the continuous case, $F_\theta$ is a smooth distribution function and $F_\theta(T_{Y,\theta})$ is uniformly distributed.  In the discontinuous case, $F_\theta$ has jump discontinuities, but it is well-known that $F_\theta(T_{Y,\theta})$ is stochastically larger than uniform.  In either case, the latter term in \eqref{eq:bound} can be bounded above by $\alpha$.  Taking supremum over $\theta \in A$ throughout \eqref{eq:bound} gives the result in \eqref{eq:validity}.  
\end{proof}

The claim that the plausibility function-based test in \eqref{eq:plausibility.test} achieves the nominal frequentist size follows as an immediate corollary of Theorem~\ref{thm:validity}.

\begin{corollary}
\label{co:size}
For any $\alpha \in (0,1)$, the size of the test \eqref{eq:plausibility.test} is no more than $\alpha$.  That is, $\sup_{\theta \in \Theta_0} \prob_\theta\{\pl_Y(\Theta_0) \leq \alpha\} \leq \alpha$.  Moreover, if $H_0$ is a point-null, so that $\Theta_0$ is a singleton, and $T_{Y,\theta}$ is a continuous random variable when $Y \sim \prob_\theta$, then the size is exactly $\alpha$.  
\end{corollary}

\begin{proof}
Apply Theorem~\ref{thm:validity} with $A = \Theta_0$.  
\end{proof}

As indicated earlier, the case where $A$ is a singleton set is an important special case for point estimation and plausibility region construction.  The result in Theorem~\ref{thm:validity} specializes nicely in this singleton case.  

\begin{theorem}
\label{thm:singleton}
\emph{(i)} If $T_{Y,\theta}$ is a continuous random variable as a function of $Y \sim \prob_\theta$, then $\pl_Y(\theta)$ is uniformly distributed.  \emph{(ii)} If $T_{Y,\theta}$ is a discrete random variable when $Y \sim \prob_\theta$, then $\pl_Y(\theta)$ is stochastically larger than uniform. 
\end{theorem}

\begin{proof}%[Proof of Theorem~\ref{thm:singleton}]
In this case, $\pl_Y(\theta) = F_\theta(T_{Y,\theta})$ so no optimization is required compared to the general $A$ case.  Therefore, the ``$\leq$'' between the second and third terms in \eqref{eq:bound} becomes an ``$=$.''  The rest of the proof goes exactly like that of Theorem~\ref{thm:validity}.  
\end{proof}  

The promised result on the frequentist coverage probability of the plausibility region $\Pi_y(\alpha)$ in \eqref{eq:plausibility.region} follows as an immediate corollary of Theorem~\ref{thm:singleton}

\begin{corollary}
\label{co:coverage}
For any $\alpha \in (0,1)$, the plausibility region $\Pi_Y(\alpha)$ has the nominal frequentist coverage probability; that is, $\prob_\theta\{\Pi_Y(\alpha) \ni \theta\} \geq 1-\alpha$.  Furthermore, the coverage probability is exactly $1-\alpha$ in case \emph{(i)} of Theorem~\ref{thm:singleton}.   
\end{corollary}

\begin{proof}
Observe that $\prob_\theta\{\Pi_Y(\alpha) \ni \theta\} = \prob_\theta\{\pl_Y(\theta) > \alpha\}$.  So according to Theorem~\ref{thm:singleton}, this probability is at least $1-\alpha$.  Moreover, in case (i) of Theorem~\ref{thm:singleton}, $\pl_Y(\theta)$ is uniformly distributed, so ``at least'' can be changed to ``equal to.'' 
\end{proof}

Theorems~\ref{thm:validity} and \ref{thm:singleton}, and their corollaries, demonstrate that the proposed method is valid for any model, any problem, and any sample size.  Compare this to the bootstrap or analytical approximations whose theoretical validity holds only asymptotically for suitably regular problems.  One could also ask about the asymptotic behavior of the plausibility function.  Such a question is relevant because exactness without efficiency may not be particularly useful, i.e., it is desirable that the method can efficiently detect wrong $\theta$ values.  Take, for example, the case where $T_{y,\theta}$ is the relative likelihood. If $Y=(Y_1,\ldots,Y_n)$ are iid $\prob_\theta$, then $-2\log T_{Y,\theta}$ is, under mild conditions, approximately chi-square distributed.  This means that the dependence of $F_\theta$ on $\theta$ disappears, asymptotically, so , for large $n$, the plausibility region \eqref{eq:plausibility.region} is similar to 
\[ \{\theta: -2\log T_{y,\theta} < \chi^2(\alpha)\}, \]
where $\chi^2(\alpha)$ is the $100\alpha$ percentile of the appropriate chi-square distribution.  Figure~\ref{fig:probit.plot} displays both of these regions and the similarity is evident.  Since the approximate plausibility region in the above display has asymptotic coverage $1-\alpha$ and is efficient in terms of volume, the efficiency conclusion carries over to the plausibility region.  %Appendix~\ref{S:asymptotics} gives a more general and more rigorous look at the asymptotic behavior of the plausibility function.  

For a more precise description of the asymptotic behavior of the plausibility function, I now present a simple but general and rigorous result.  Again, since the plausibility function-based methods are valid for all fixed sample sizes, the motivation for this asymptotic investigation is efficiency.  Let $Y=(Y_1,\ldots,Y_n)$ be iid $\prob_{\theta^\star}$.  Suppose that the loss function is additive, i.e., $\ell(y,\theta) = \sum_{i=1}^n h(y_i,\theta)$, and the function $h$ is such that $H(\theta) = \E_{\theta^\star}\{h(Y_1,\theta)\}$ exists, is finite for all $\theta$, and has a unique minimum at $\theta=\theta^\star$.  Also assume that, for sufficiently large $n$, the distribution of $T_{Y,\theta}$ under $\prob_\theta$ has no atom at zero.  Write $\pl_n$ for the plausibility function $\pl_Y=\pl_{(Y_1,\ldots,Y_n)}$.  

\begin{theorem}
\label{thm:limit}
Under the conditions in the previous paragraph, $\pl_n(\theta) \to 0$ with $\prob_{\theta^\star}$-probability~1 for any $\theta \neq \theta^\star$.  
\end{theorem}

\begin{proof}
Let $\hat\theta=\hat\theta(Y)$ denote the loss minimizer.  Then $\ell(Y,\theta^\star) \geq \ell(Y,\hat\theta)$, so 
\begin{align*}
T_{Y,\theta} & = \exp[-\{\ell(Y,\theta) - \ell(Y,\hat\theta)\}] \\
& \leq \exp[-\{\ell(Y,\theta) - \ell(Y,\theta^\star)\}] \\
%& = \exp\Bigl[-\sum_{i=1}^n \{f(Y_i,\theta) - f(Y_i, \theta^\star)\} \Bigr] \\
& = \exp[-n\{ H_n(\theta) - H_n(\theta^\star)\} ], 
\end{align*}
where $H_n(\theta) = n^{-1} \sum_{i=1}^n h(Y_i,\theta)$ is the empirical version of $H(\theta)$.  Since $F_\theta(\cdot)$ is non-decreasing, 
\[ \pl_n(\theta) = F_\theta(T_{Y,\theta}) \leq F_\theta(e^{-n\{H_n(\theta) - H_n(\theta^\star)\}}). \]
By the assumptions on the loss and the law of large numbers, with $\prob_{\theta^\star}$-probability~1, there exists $N$ such that $H_n(\theta) - H_n(\theta^\star) > 0$ for all $n \geq N$.  Therefore, the exponential term in the above display vanishes with $\prob_{\theta^\star}$-probability~1.  Since $T_{Y,\theta}$ has no atom at zero under $\prob_\theta$, the distribution function $F_\theta(t)$ is continuous at $t=0$ and satisfies $F_\theta(0)=0$.  It follows that $\pl_n(\theta) \to 0$ with $\prob_{\theta^\star}$-probability~1.  %Compare this to a simple well-known result for the likelihood function, e.g., or \citet[][Theorem~6.3.2]{lehmann.casella.1998}.  
\end{proof}
 
The conclusion of Theorem~\ref{thm:limit} is that the plausibility function will correctly distinguish between the true $\theta^\star$ and any $\theta \neq \theta^\star$ with probability~1 for large $n$.  In other words, if $n$ is large, then the plausibility region will not contain points too far from $\theta^\star$, hence efficiency.  For the case of the relative likelihood, the difference $H_n(\theta)-H_n(\theta^\star)$ in the proof converges to the Kullback--Leibler divergence of $\prob_\theta$ from $\prob_{\theta^\star}$, which is strictly positive under the uniqueness condition on $\theta^\star$, i.e., identifiability.  

It is possible to strengthen the convergence result in Theorem~\ref{thm:limit}, at least for the relative likelihood case, with the use of tools from the theory of empirical processes, as in \citet{wongshen1995}.  However, I have found that this approach also requires some uniform control on the small quantiles of the distribution $F_\theta$ for $\theta$ away from $\theta^\star$.  These quantiles are difficult to analyze, so more work is needed here.

\subsection{Implementation}
\label{SS:computation}

Evaluation of $F_\theta(T_{y,\theta})$ is crucial to the proposed methodology.  In some problems, it may be possible to derive the distribution $F_\theta$ in either closed-form or in terms of some functions that can be readily evaluated, but such problems are rare.  So, numerical methods are needed to evaluate the plausibility function and, here, I present a simple Monte Carlo approximation of $F_\theta$.  See, also, Remark~\ref{re:monte.carlo} in Section~\ref{S:remarks}. 

To approximate $F_\theta(T_{y,\theta})$, where $Y=y$ is the observed sample, first choose a large number $M$; unless otherwise stated, the examples herein use $M=50,000$, which is conservative.  Then construct the following root-$M$ consistent estimate of $F_\theta(T_{y,\theta})$:
\begin{equation}
\label{eq:monte.carlo}
\widehat F_\theta(T_{y,\theta}) = \frac1M \sum_{m=1}^M I\{T_{Y^{(m)},\theta} \leq T_{y,\theta}\}, \quad Y^{(1)},\ldots,Y^{(M)} \iid \prob_\theta. 
\end{equation}
This strategy can be performed for any choice of $\theta$, so we may consider $\widehat F_\theta(T_{y,\theta})$ as a function of $\theta$.  If necessary, the supremum over a set $A \subset \Theta$ can be evaluated using a standard optimization package; in my experience, the {\tt optim} function in R works well.  To compute a plausibility interval, solutions to the equation $\widehat F_\theta(T_{y,\theta}) = \alpha$ are required.  These can be obtained using, for example, standard bisection or stochastic approximation \citep{garthwaite.buckland.1992}.  

An interesting question is if, and under what conditions, the distribution function $F_\theta$ does not depend on $\theta$.  Indeed, if $F_\theta$ is free of $\theta$, then there is no need to simulate new $Y^{(m)}$'s for different $\theta$'s---the same Monte Carlo sample can be used for all $\theta$---which amounts to substantial computational savings.  Next I describe a general context where this $\theta$-independence can be discussed.  

Let $\G$ be a group of transformations $g:\YY \to \YY$, and let $\Gbar$ be a corresponding group of transformations $\gbar: \Theta \to \Theta$ defined by the invariance condition:
\begin{equation}
\label{eq:invariant}
\text{if $Y \sim \prob_\theta$, then $gY \sim \prob_{\gbar\theta}$},
\end{equation}
where, e.g., $gy$ denotes the image of $y$ under transformation $g$.  Note that $g$ and $\gbar$ are tied together by the relation \eqref{eq:invariant}.  Models that satisfy \eqref{eq:invariant} are called group transformation models.  The next result, similar to Corollary~1 in \citet{spjotvoll1972}, shows that $F_\theta$ is free of $\theta$ in group transformation models when $T_{y,\theta}$ has a certain invariance property.  

\begin{theorem}
\label{thm:efficiency}
Suppose \eqref{eq:invariant} holds for groups $\G$ and $\Gbar$ as described above.  If $\Gbar$ is transitive on $\Theta$ and $T_{y,\theta}$ satisfies
\begin{equation}
\label{eq:T.invariant}
T_{gy,\gbar\theta} = T_{y,\theta} \quad \text{for all $y \in \YY$ and $g \in \G$},
\end{equation}
then the distribution function $F_\theta$ in \eqref{eq:cdf} does not depend on $\theta$.   
\end{theorem}

\begin{proof}%[Proof of Theorem~\ref{thm:efficiency}]
For $Y \sim \prob_\theta$, pick any fixed $\theta_0$ and choose corresponding $g,\gbar$ such that $\theta = \gbar\theta_0$; such a choice is possible by transitivity.  Let $Y_0 \sim \prob_{\theta_0}$, so that $gY_0$ also has distribution $\prob_\theta$.  Since $T_{gY_0,\gbar\theta_0} = T_{Y_0,\theta_0}$, by \eqref{eq:T.invariant}, it follows that $T_{Y_0,\theta_0}$ has distribution free of $\theta$.  
\end{proof}

For a given function $T_{y,\theta}$, condition \eqref{eq:T.invariant} needs to be checked.  For the loss function-based description of $T_{y,\theta}$, if $\ell(y,\theta)$ is invariant with respect to $\G$, i.e., 
\[ \ell(gy, \gbar \theta) = \ell(y,\theta), \quad \forall \; (g,\gbar), \quad \forall \; (y,\theta), \]
and if the loss minimizer $\hat\theta=\hat\theta(y)$ is equivariant, i.e., 
\[ \hat\theta(gy) = \gbar \hat\theta(y), \quad \forall \; (g,\gbar), \quad \forall \; y, \]
then \eqref{eq:T.invariant} holds.  For the special case where $\ell(y,\theta)$ is the negative log-likelihood, so that $T_{y,\theta}$ is the relative likelihood \eqref{eq:relative.likelihood}, we have the following result.  

\begin{corollary}
\label{co:efficiency}
Suppose the model has dominating measure relatively invariant with respect to $\G$.  Then \eqref{eq:invariant} holds, and the relative likelihood $T_{Y,\theta}$, in \eqref{eq:relative.likelihood}, satisfies \eqref{eq:T.invariant}.  Therefore, if $\Gbar$ is transitive, the distribution function $F_\theta$ in \eqref{eq:cdf} does not depend on $\theta$.  
\end{corollary}

\begin{proof}
\citet[][Theorem~3.1]{eaton1989} establishes \eqref{eq:invariant}.  Moreover, \citet[][pp.~46--47]{eaton1989} argues that, under the stated conditions, the likelihood satisfies $L_y(\theta) = L_{gy}(g\theta) \chi(g)$, for a multiplier $\chi$ that does not depend on $y$ or $\theta$.  This invariance result immediately implies \eqref{eq:T.invariant} for the relative likelihood $T_{Y,\theta}$.  Now apply Theorem~\ref{thm:efficiency}.  
\end{proof}

%Note that it may not be necessary to simulate a full data set $Y^{(m)}$ at each Monte Carlo step.  Indeed, if $T_{Y,\theta}$ is a variant of the likelihood, then it is enough to simulate a sufficient statistic.  When $F=F_\theta$ does not depend on $\theta$ (see Theorem~\ref{thm:efficiency}), then it is not necessary to simulate the $Y^{(m)}$ values for each value of $\theta$, i.e., the same simulations can be used to evaluate $F(T_{y,\theta})$ for each $\theta$.  When different simulations are needed for each $\theta$, it may be possible to implement an importance weight resampling procedure to reduce the time spent on simulations.  

\subsection{Examples}
\label{SS:examples1}

\begin{example}
\label{ex:binomial}
Inference on the success probability $\theta$ based on a sample $Y$ from a binomial distribution is a fundamental problem in statistics.  For this problem, \citet{bcd2001, bcd2002} showed that the widely-used Wald confidence interval often suffers from strikingly poor frequentist coverage properties, and that other intervals can be substantially better in terms of coverage.  In the present context, the relative likelihood is given by 
\[ T_{y,\theta} = \Bigl( \frac{n\theta}{y} \Bigr)^y \Bigl( \frac{n-n\theta}{n-y} \Bigr)^{n-y}. \]
For given $(n,y)$, one can exactly evaluate $\pl_y(\theta) = \prob_\theta(T_{Y,\theta} \leq T_{y,\theta})$, where $Y \sim \bin(n,\theta)$, numerically, using the binomial mass function.  Given $\pl_y(\theta)$, the $100(1-\alpha)$\% plausibility interval for $\theta$ can be found by solving the equation $\pl_y(\theta) = \alpha$ numerically.  Figure~\ref{fig:binomial}(a) shows a plot of the plausibility function for data $(n,y) = (25, 15)$.  As expected, at $\hat\theta=15/25=0.6$, the plausibility function is unity.  The steps in the plausibility function are caused by the discreteness of the underlying binomial distribution.  The figure also shows (in gray) an approximation of the plausibility function obtained by Monte Carlo sampling ($M=1000$) from the binomial distribution, as in \eqref{eq:monte.carlo}, and the exact and approximation plausibility functions are almost indistinguishable.  By the general theory above, this $100(1-\alpha)$\% plausibility interval has guaranteed coverage probability $1-\alpha$.  However, the discreteness of the problem implies that the coverage is conservative.  A plot of the coverage probability, as a function of $\theta$, for $n=50$, is shown in Figure~\ref{fig:binomial}(b), confirming the claimed conservativeness (up to simulation error).  Of the intervals considered in \citet{bcd2001}, only the Clopper--Pearson interval has guaranteed 0.95 coverage probability.  The plausibility interval here is clearly more efficient, particularly for $\theta$ near 0.5.   
\end{example}

\begin{figure}
\begin{center}
\subfigure[Plausibility function]{\scalebox{0.55}{\includegraphics{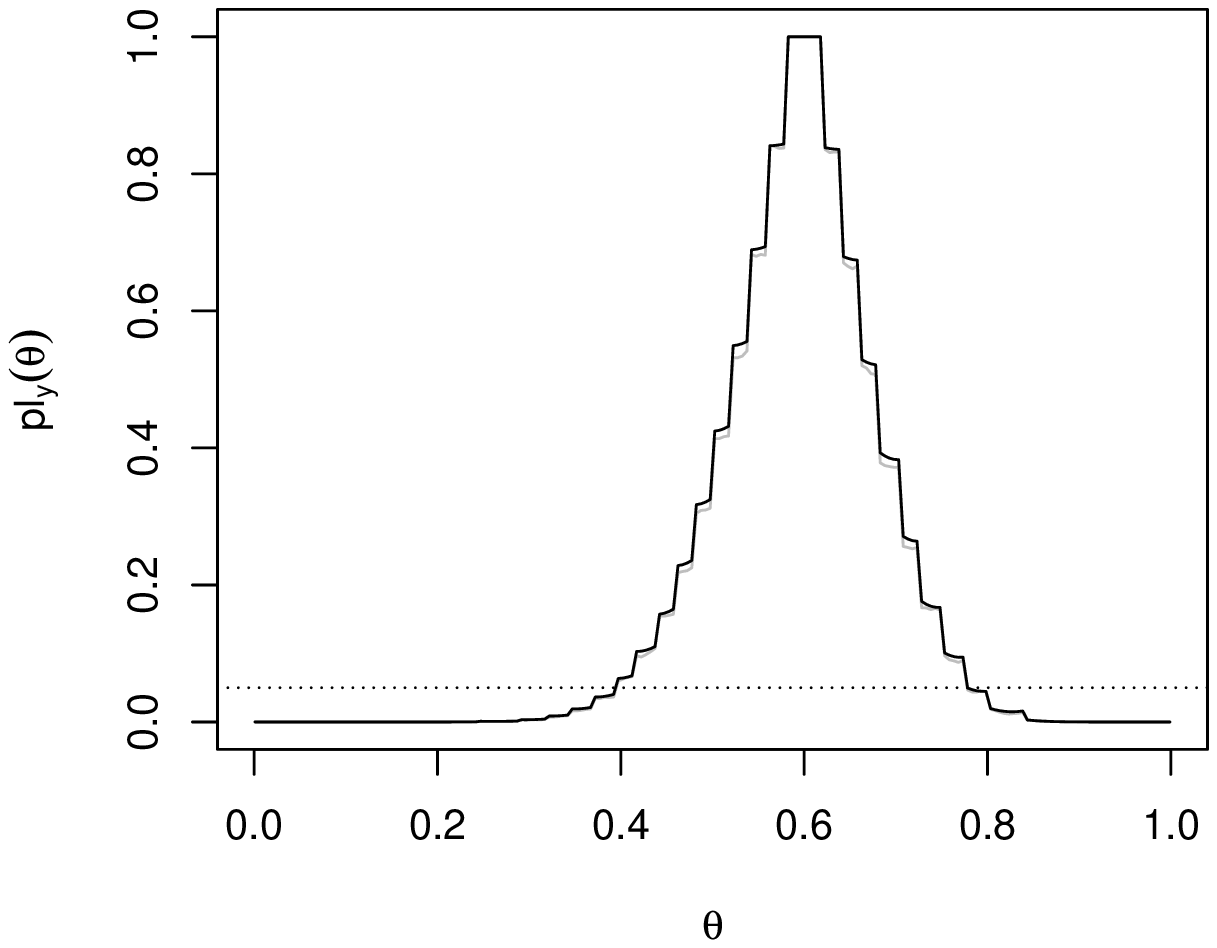}}}
\subfigure[Coverage probability]{\scalebox{0.55}{\includegraphics{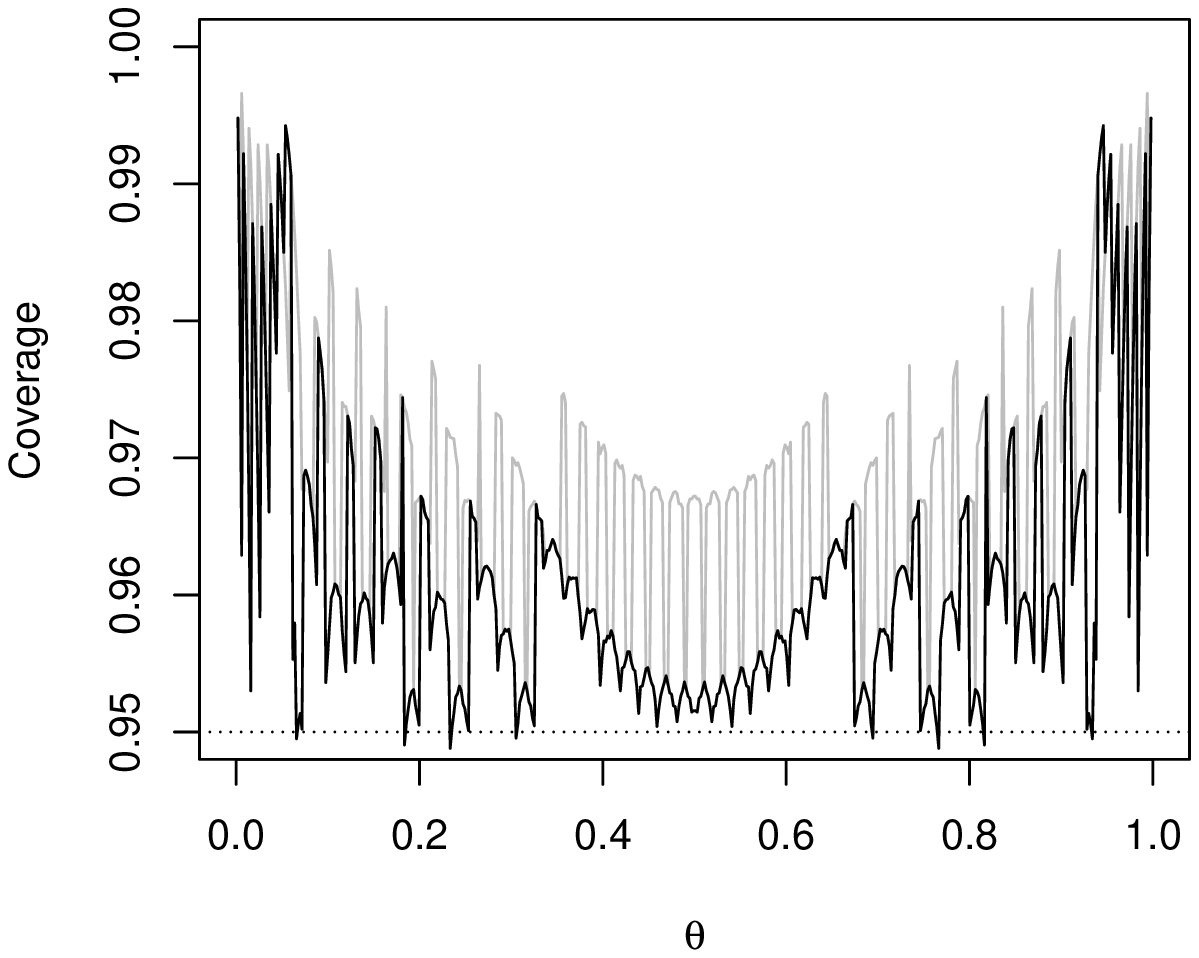}}}
\end{center}
\caption{Panel~(a): Exact (black) and approximate (gray) plausibility functions for a data set with $(n,y)=(25, 15)$. Panel~(b): Coverage probability of the 95\% plausibility (black) and Clopper--Pearson (gray) intervals, as a function of $\theta$, for $n=50$.}
\label{fig:binomial}
\end{figure}

\begin{example}
\label{ex:lindley}
Let $Y_1,\ldots,Y_n$ be independent samples from a distribution with density $p_\theta(y) = \theta^2(\theta + 1)^{-1}(y+1) e^{-\theta y}$, for $y,\theta > 0$.  This non-standard distribution, a mixture of a gamma and an exponential density, appears in \citet{lindley1958}.  In this case, $\hat\theta = \{ 1-\ybar + (\ybar^2 + 6\ybar + 1)^{1/2}\} / 2\ybar$, and the relative likelihood is 
\[ T_{y,\theta} = (\theta / \hat\theta)^{2n} \{ (\hat\theta + 1) / (\theta + 1) \}^n e^{n\ybar(\hat\theta - \theta)}. \]
For illustration, I compare the coverage probability of the 95\% plausibility interval versus those based on standard asymptotic normality of $\hat\theta$ and a corresponding parametric bootstrap.  With 1000 random samples of size $n=50$ from the distribution above, with $\theta=1$, the estimated coverage probabilities are 0.949, 0.911, and 0.942 for plausibility, asymptotic normality, and bootstrap, respectively.  The plausibility interval hits the desired coverage probability on the nose, while the other two, especially the asymptotic normality interval, fall a bit short.  
\end{example}

\begin{example}
\label{ex:gamma}
Consider an iid sample $Y_1,\ldots,Y_n$ from a gamma distribution with unknown shape $\theta_1$ and scale $\theta_2$.  Maximum likelihood estimation of $(\theta_1,\theta_2)$ in the gamma problem has an extensive body of literature, e.g., \citet{greenwood.durand.1960}, \citet{harter.moore.1965}, and \citet{bowman.shenton.1988}.  In this case, the maximum likelihood estimate has no closed-form expression, but the relative likelihood can be readily evaluated numerically and the plausibility function can be found via \eqref{eq:monte.carlo}.  For illustration, consider the data presented in \citet{fraser.reid.wong.1997} on the survival times of $n=20$ rats exposed to a certain amount of radiation.  A plot of the 90\% plausibility region for $\theta=(\theta_1,\theta_2)$ is shown in Figure~\ref{fig:gamma.plot}.  A Bayesian posterior sample is also shown, based on Jeffreys prior, along with a plot of the 90\% confidence ellipse based on asymptotic normality of the maximum likelihood estimate.  Since $n$ is relatively small, the shape the Bayes posterior is non-elliptical.  The plausibility region captures the non-elliptical shape, and has roughly the same center and size as the maximum likelihood region.  Moreover, the plausibility region has exact coverage.  
\end{example}

\begin{figure}
\begin{center}
\scalebox{0.55}{\includegraphics{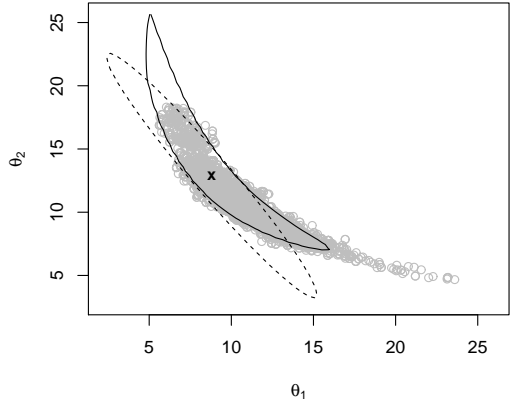}}
\end{center}
\caption{Solid line gives the 90\% plausibility region for $\theta=(\theta_1,\theta_2)$ in the gamma example; {\sf X} marks the maximum likelihood estimate; dashed line gives the 90\% confidence region for $\theta$ based on asymptotic normality of the maximum likelihood estimator; the gray points are samples from the (Jeffreys' prior) Bayesian posterior distribution. }
\label{fig:gamma.plot}
\end{figure}

\begin{example}
\label{ex:probit1}
Consider a binary response variable $Y$ that depends on a set of covariates $x=(1,x_1,\ldots,x_p)^\top \in \RR^{p+1}$.  An important special case is the probit regression model, with likelihood $L_y(\theta) = \prod_{i=1}^n \Phi(x_i^\top\theta)^{y_i}\{1-\Phi(x_i^\top\theta)\}^{1-y_i}$, where $y_1,\ldots,y_n$ are the observed binary response variables, $x_i = (1,x_{i1},\ldots,x_{ip})^\top$ is a vector of covariates associated with $y_i$, $\Phi$ is the standard Gaussian distribution function, and $\theta=(\theta_0,\theta_1,\ldots,\theta_p)^\top \in \Theta$ is an unknown coefficient vector.  This likelihood function can be maximized to obtain the maximum likelihood estimate $\hat\theta$ and, hence, the relative likelihood $T_{y,\theta}$ in \eqref{eq:relative.likelihood}.  Then the plausibility function $\pl_y(\theta)$ can be evaluated as in \eqref{eq:monte.carlo}.  

For illustration, I consider a real data set with a single covariate ($p=1$).  The data, presented in Table~8.4 in \citet[][p.~252]{ghosh-etal-book}, concerning the relationship between exposure to choleric acid and the death of mice.  In particular, the covariate $x$ is the acid dosage and $y=1$ if the exposed mice dies and $y=0$ otherwise.  Here a total of $n=120$ mice are exposed, ten at each of the twelve dosage levels.  Figure~\ref{fig:probit.plot} shows the 90\% plausibility region for $\theta=(\theta_0,\theta_1)^\top$.  %The center point is the maximum likelihood estimate $\hat\theta = (-2.13,8.73)^\top$.  
For comparison, the 90\% confidence region based on the asymptotic normality of $\hat\theta$ is also given.  In this case, the plausibility and confidence regions are almost indistinguishable, likely because $n$ is relatively large.  The 0.9 coverage probability of the plausibility region is, however, guaranteed and its similarity to the classical region suggests that it is also efficient.  
\end{example}

\begin{figure}
\begin{center}
\scalebox{0.55}{\includegraphics{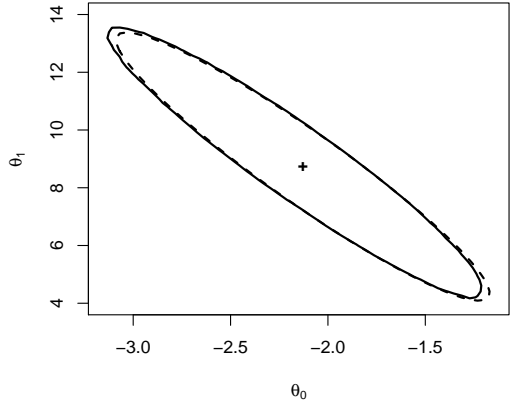}}
\end{center}
\caption{Solid line gives the 90\% plausibility region for $\theta=(\theta_0,\theta_1)$ in the binary regression example; dashed line gives the 90\% confidence region for $\theta$ based on asymptotic normality of the maximum likelihood estimator.  }
\label{fig:probit.plot}
\end{figure}

%\begin{example}
%\label{ex:correlation1}
%Let $Y=\{(Y_{i1},Y_{i2}): i=1,\ldots,n\}$ be a sample from a standard bivariate Gaussian distribution with correlation coefficient $\theta$; the more general case is considered in Example~\ref{ex:correlation2}.  Figure~\ref{fig:probit.plot}(b) shows a plot of the plausibility function $\pl_y(\theta)$, where $y$ is the observed sample of size $n=25$ from the standard bivariate Gaussian distribution with correlation $\theta^\star = 0.1$.  The peak is at the maximum likelihood estimate $\hat\theta=-0.035$, and the $\alpha=0.05$ level set gives the 95\% plausibility interval for $\theta$.  The dashes on the $\theta$-axis mark the 95\% confidence interval based on Fisher's exact distribution of the sample correlation.  In this case, both intervals have exact coverage, but the plausibility interval is considerably shorter.   
%\end{example}

%\begin{figure}
%\begin{center}
%\scalebox{0.75}{\includegraphics{gauss_cor_plot}}
%\end{center}
%\caption{Plausibility function in Example~\ref{ex:correlation1}.  Asterisk marks $\theta^\star=0.1$; vertical line marks $\hat\theta=-0.035$; horizontal line at $\alpha=0.05$ determines the cutoffs for the 95\% plausibility interval.  Dashes on the $\theta$-axis mark endpoints of Fisher's exact interval.}
%\label{fig:gauss.cor.plot}
%\end{figure}

\section{Marginal plausibility functions}
\label{S:nuisance}

\subsection{Construction}

In many cases, $\theta$ can be partitioned as $\theta = (\psi,\lambda) \in \Psi \times \Lambda$ where $\psi$ is the parameter of interest and $\lambda$ is a nuisance parameter.  For example, $\psi$ could be just a component of the parameter vector $\theta$ or, more generally, $\psi$ is some function of $\theta$.  In such a case, the approach described above can be applied with special kinds of sets, e.g., $A = \{\theta=(\psi,\lambda): \lambda \in \Lambda\}$, to obtain marginal inference for $\psi$.  However, it may be easier to interpret a redefined \emph{marginal} plausibility function.  The natural extension to the methodology presented in Section~\ref{S:plausibility} is to consider some loss function $\ell(y,\psi)$ that does not directly consider the nuisance parameter $\lambda$, and construct the function $T_{y,\psi}$, depending only on the interest parameter $\psi$, just as before.  For example, taking $\ell(y,\psi)=-\sup_\lambda \log L(\psi, \lambda)$ to be the negative profile likelihood corresponds to replacing the relative likelihood \eqref{eq:relative.likelihood} with the relative profile likelihood 
\begin{equation}
\label{eq:profile.likelihood}
T_{y,\psi} = L_y(\psi,\hat\lambda_\psi) \, / \, L_y(\hat\psi,\hat\lambda), 
\end{equation}
where $\hat\lambda_\psi$ is the conditional maximum likelihood estimate of $\lambda$ when $\psi$ is fixed, and $(\hat\psi, \hat\lambda)$ is a global maximizer of the likelihood.  As before, other choices of $T_{y,\psi}$ are possible, but \eqref{eq:profile.likelihood} is an obvious choice and shall be my focus in what follows.  

If the distribution of $T_{Y,\psi}$, as a function of $Y \sim \prob_{\psi,\lambda}$, does not depend on $\lambda$, then the development in the previous section carries over without a hitch.  That is, one can define the distribution function $F_\psi$ of $T_{Y,\psi}$ and construct a marginal plausibility function just as before:
\begin{equation}
\label{eq:marginal.plausibility}
\mpl_y(A) = \sup_{\psi \in A} F_\psi( T_{y,\psi}), \quad A \subseteq \Psi, 
\end{equation}
This function can, in turn, be used exactly as in Section~\ref{SS:inference} for inference on the parameter $\psi$ of interest, e.g., a $100(1-\alpha)$\% marginal plausibility region for $\psi$ is 
\begin{equation}
\label{eq:m.plaus.region}
\{\psi: \mpl_y(\psi) > \alpha\}. 
\end{equation}
The distribution function $F_\psi$ can be approximated via Monte Carlo just as in Section~\ref{SS:computation}; see \eqref{eq:marginal.monte.carlo} below.  Unfortunately, checking that $F_\psi$ does not depend on the nuisance parameter $\lambda$ is a difficult charge in general.  This issue is discussed further below.

\subsection{Theoretical considerations}

It is straightforward to verify that the sampling distribution properties (Theorems~\ref{thm:validity}--\ref{thm:singleton}) of the plausibility function carry over exactly in this more general case, provided that $T_{Y,\psi}$ in \eqref{eq:profile.likelihood} has distribution free of $\lambda$, as a function of $Y \sim \prob_{\psi,\lambda}$.  Consequently, the basic properties of the plausibility regions and tests (Corollaries~\ref{co:size}--\ref{co:coverage}) also hold in this case.  It is rare, however, that $T_{Y,\psi}$ can be written in closed-form, so checking if its distribution depends on $\lambda$ can be challenging.  

Following the ideas in Corollary~\ref{co:efficiency}, it is natural to consider models having a special structure.  The particular structure of interest here is that where, for each fixed $\psi$, $\prob_{\psi,\lambda}$ is a transformation model with respect to $\lambda$.  That is, there exists associated groups of transformations, namely, $\G$ and $\Gbar$, such that 
\[ \text{if $Y \sim \prob_{\psi, \lambda}$, then $gY \sim \prob_{\psi, \gbar\lambda}$}, \quad \text{for all $\psi$}. \]
This is called a composite transformation model; \citet{bn1988} gives several examples, and Example~\ref{ex:correlation} below gives another.  For such models, it follows from the argument in the proof of Theorem~\ref{thm:efficiency} that, if the loss $\ell(y,\psi)$ is invariant to the group action, i.e., if $\ell(gy, \psi) = \ell(y,\psi)$ for all $y$ and $g$, then the corresponding $T_{Y,\psi}$ has distribution that does not depend on $\lambda$.  Therefore,  inference based on the marginal plausibility function $\mpl_y$ is exact in these composite transformation models.  See Examples~\ref{ex:simple.gaussian} and \ref{ex:correlation}.

What if the problem is not a composite transformation model?  In some cases, it is possible to show directly that the distribution of $T_{Y,\psi}$ does not depend on $\lambda$ (see Examples~\ref{ex:simple.gaussian}--\ref{ex:correlation} and \ref{ex:random.effects}) but, in general, this seems difficult.  Large-sample theory can, however, provide some guidance.  For example, if $Y=(Y_1,\ldots,Y_n)$ is a vector of iid samples from $\prob_{\psi,\lambda}$, then it can be shown, under certain standard regularity conditions, that $-2\log T_{Y,\psi}$ is asymptotically chi-square, \emph{for all values of $\lambda$} \citep{bickel1998, murphy.vaart.2000}.  Similar conclusions can be reached for the case where $T_{Y,\psi}$ is a conditional likelihood \citep{andersen1971}.  This suggests that, at least for large $n$, $\lambda$ has a relatively weak effect on the sampling distribution of $T_{Y,\psi}$.   This, in turn, suggests the following intuition: since $\lambda$ has only a minimal effect, construct a marginal plausibility function for $\psi$, by fixing $\lambda$ to be at some convenient value $\lambda_0$.  In particular, a Monte Carlo approximation of $F_\psi$ is as follows:
\begin{equation}
\label{eq:marginal.monte.carlo}
\widehat F_\psi(T_{y,\psi}) = \frac1M \sum_{m=1}^M I\{T_{Y^{(m)},\psi} \leq T_{y,\psi}\}, \quad Y^{(1)},\ldots,Y^{(M)} \iid \prob_{\psi,\lambda_0}.
\end{equation}
Numerical justification for this approximation is provided in Example~\ref{ex:gamma.mean}.  

One could also consider different choices of $T_{Y,\psi}$ that might be less sensitive to the choice of $\lambda$.  For example, the Bartlett correction to the likelihood ratio or the signed likelihood root often have faster convergence to a limiting distribution, suggesting less dependence on $\lambda$ \citep{skovgaard2001, bn.hall.1986, bn1986}.  Such quantities have also been used in conjunction with bootstrap/Monte Carlo schemes that avoid use of the approximate limiting distribution; see \citet{diciccio.martin.stern.2001} and \citet{lee.young.2005}.  These adjustments, special cases of the general program here, did not appear to be necessary in the examples considered below.  However, further work is needed along these lines, particularly in the case of high-dimensional $\lambda$.

\subsection{Examples}
\label{SS:examples2}

\begin{example}
\label{ex:simple.gaussian}
For a simple illustrative example, let $Y_1,\ldots,Y_n$ independent with distribution $\nm(\psi, \lambda)$, where $\theta=(\psi,\lambda)$ is completely unknown, but only the mean $\psi$ is of interest.  In this case, the relative profile likelihood is 
\[ T_{Y,\psi} = \bigl\{1 + n(\Ybar-\psi)^2/S^2 \bigr\}^{-n/2}, \]
where $S^2 = \sum_{i=1}^n (Y_i - \Ybar)^2$ is the usual residual sum-of-squares.  Since $T_{Y,\psi}$ is a monotone decreasing function of the squared $t$-statistic, it is easy to see that the marginal plausibility interval \eqref{eq:m.plaus.region} for $\psi$ is exactly the textbook $t$-interval.  Exactness and efficiency of the marginal plausibility interval follow from the well-known results for the $t$-interval.  %By similar arguments, analogous plausibility interval results obtain for the individual slope coefficients in a Gaussian linear regression setting.    
\end{example}

\begin{example}
\label{ex:nonparametric}
Suppose that $Y_1,\ldots,Y_n$ are independent real-valued observations from an unknown distribution $\prob$, a nonparametric problem.  Consider the so-called empirical likelihood ratio, given by $n^n \prod_{i=1}^n \prob(\{Y_i\})$, where $\prob$ ranges over all probability measures on $\RR$ \citep{owen1988}.  Here interest is in a functional $\psi=\psi(\prob)$, namely the $100p$th quantile of $\prob$, where $p \in (0,1)$ is fixed.  \citet[][Theorem~5]{wasserman1990b} shows that 
\[ T_{Y,\psi} = \Bigl( \frac{p}{r} \Bigr)^r \Bigl( \frac{1-p}{n-r} \Bigr)^{n-r}, \quad \text{where} \quad r = \begin{cases} \#\{i: Y_i \leq \psi\} & \text{if $\psi < \hat\psi$} \\ np & \text{if $\psi = \hat\psi$} \\ \#\{i: Y_i < \psi\} & \text{if $\psi > \hat\psi$}, \end{cases} \]
and $\hat\psi$ is the $100p$th sample quantile.  The distribution of $T_{Y,\psi}$ depends on $\prob$ only through $\psi$, so the marginal plausibility function is readily obtained via basic Monte Carlo.  In fact, it is now essentially a binomial problem, like in Example~\ref{ex:binomial}.  %However, the discreteness will make the resulting inference somewhat conservative.  %Regarding efficiency, I report on a small simulation study.  For standard normal, exponential, and Student-t ($\text{df}=3$) distributions, 1000 samples of size $n=100$ are taken and, for each, a 95\% plausibility interval for the $75^{\text{th}}$ percentile is calculated.  The estimated coverage probabilities are 0.957, 0.962, and 0.959, respectively.  This suggests, as expected, that the  plausibility intervals are conservative.  Similar results are seen for different percentiles and for smaller/larger samples sizes.  
\end{example}

\begin{example}
\label{ex:correlation}
Consider a bivariate Gaussian distribution with all five parameters unknown.  That is, the unknown parameter is $\theta = (\psi,\lambda)$, where the correlation coefficient $\psi$ is the parameter of interest, and $\lambda = (\mu_1,\mu_2,\sigma_1,\sigma_2)$ is the nuisance parameter.  From the calculations in \citet{sun.wong.2007}, the relative profile likelihood is 
\[ T_{Y,\psi} = \bigl\{ (1-\psi^2)^{1/2}(1-\hat\psi^2)^{1/2} \, / \, (1-\psi\hat\psi) \bigr\}^n, \]
where $\hat\psi$ is the sample correlation coefficient.  It is clear from the previous display and basic properties of $\hat\psi$ that the distribution of $T_{Y,\psi}$ is free of $\lambda$; this fact could also have been deduced directly from the problem's composite transformation structure.  Therefore, for the Monte Carlo approximation in \eqref{eq:monte.carlo}, data can be simulated from the bivariate Gaussian distribution with any convenient choice of $\lambda$.  

For illustration, I replicate a simulation study in \citet{sun.wong.2007}.  Here 10,000 samples of size $n=10$ from a bivariate Gaussian distribution with $\lambda = (1,2,1,3)$ and various $\psi$ values.  Coverage probabilities of the 95\% plausibility intervals \eqref{eq:m.plaus.region} are displayed in Table~\ref{tab:correlation1}.  For comparison, several other methods are considered: 
\begin{itemize}
\item Fisher's interval, based on approximate normality of $z=\frac12 \log\{(1 + \hat\psi)/(1-\hat\psi)\}$;
\vspace{-2mm}
\item a modification of Fisher's $z$, due to \citet{hotelling1953}, based on approximate normality of 
\[ z_4 = z - \frac{3z+\hat\psi}{4(n-1)} - \frac{23z + 33\hat\psi - 5\hat\psi^2}{96(n-1)^2}; \]
%\vspace{-2mm}
\item third-order approximate normality of $R^* = R - R^{-1} \log(RQ^{-1})$, where $R$ is a signed log-likelihood root and $Q$ is a measure of maximum likelihood departure, with expressions for $R$ and $Q$ worked out in \citet{sun.wong.2007};  
\vspace{-2mm}
\item standard parametric bootstrap percentile confidence intervals based on the sample correlation coefficient, with 5000 bootstrap samples.  
\end{itemize}
In this case, based on the first three digits, $z$, $z_4$, and $R^*$ perform reasonably well, but the parametric bootstrap intervals suffer from under-coverage near $\pm 1$.  The plausibility intervals are quite accurate across the range of $\psi$ values.   
\end{example}

\begin{table}
\begin{center}
\begin{tabular}{cccccc}
\hline
& \multicolumn{5}{c}{Correlation, $\psi$} \\
\cline{2-6}
Method & $-0.9$ & $-0.5$ & $0.0$ & $0.5$ & $0.9$ \\
\hline
$z$ & 0.9527 & 0.9525 & 0.9500 & 0.9517 & 0.9542 \\
$z_4$ & 0.9499 & 0.9509 & 0.9494 & 0.9502 & 0.9516 \\
%$z_{hr}$ & 0.9643 & 0.9577 & 0.9550 & 0.9577 & 0.9610 \\
%$f_3$ & 0.9507 & 0.9514 & 0.9517 & 0.9526 & 0.9506 \\
$R^*$ & 0.9488 & 0.9500 & 0.9517 & 0.9508 & 0.9492 \\
PB & 0.9385 & 0.9425 & 0.9453 & 0.9438 & 0.9411 \\
MPL & 0.9492 & 0.9496 & 0.9502 & 0.9505 & 0.9509 \\
\hline
\end{tabular}
\end{center}
\caption{Estimated coverage probabilities of 95\% intervals for $\psi$ in the Example~\ref{ex:correlation} simulation.  First three rows are taken from Table~1 in \citet{sun.wong.2007}.  Last two rows correspond to the parametric bootstrap and marginal plausibility intervals, respectively.}  
\label{tab:correlation1}
\end{table}

\begin{example}
\label{ex:gamma.mean}
Consider a gamma distribution with mean $\psi$ and shape $\lambda$; that is, the density is $p_\theta(y) = \Gamma(\lambda)^{-1}(\lambda/\psi)^\lambda y^{\lambda-1} e^{-\lambda y/\psi}$, where $\theta = (\psi,\lambda)$.  The goal is to make inference on the mean $\psi$.   Likelihood-based solutions to this problem are presented in \citet{grice.bain.1980}, \citet{fraser.reid.1989}, \citet{fraser.reid.wong.1997}.  In the present context, it is straightforward to evaluate the relative profile likelihood $T_{Y,\psi}$ in \eqref{eq:profile.likelihood}.  However, it is apparently difficult to check if the distribution function $F_\psi$ of $T_{Y,\psi}$ depends on nuisance shape parameter $\lambda$.  So, following the general intuition above, I shall assume that it has a negligible effect and fix $\lambda \equiv 1$ in the Monte Carlo step.  That is, the Monte Carlo samples, $Y^{(1)},\ldots,Y^{(M)}$, in \eqref{eq:marginal.monte.carlo}, are each iid samples of size $n$ taken from a gamma distribution with mean $\psi$ and shape $\lambda_0=1$.  That the results are robust to fixing $\lambda_0=1$ in large samples is quite reasonable, but what about in small samples?  It would be comforting if the distribution of the relative profile likelihood $T_{Y,\psi}$ were not sensitive to the underlying value of the shape parameter $\lambda$.  Monte Carlo estimates of the distribution function of $T_{Y,\psi}$ are shown in Figure~\ref{fig:gamma.mean} for $n=10$, $\psi=1$, and a range of $\lambda$ values.  It is clear that the distribution is not particularly sensitive to the value of $\lambda$, which provides comfort in fixing $\lambda_0=1$.  Similar comparisons hold for $\psi$ different from unity.  For further justification for fixing $\lambda_0=1$, I computed the coverage probability for the 95\% marginal plausibility interval for $\psi$, based on fixed $\lambda_0=1$ in \eqref{eq:marginal.monte.carlo}, over a range of true $(\psi,\lambda)$ values; in all cases, the coverage is within an acceptable range of 0.95.  

Here I reconsider the data on survival times in Example~\ref{ex:gamma} above.  Table~\ref{tab:gamma.mean} shows 95\% intervals for $\psi$ based on four different methods: the classical first-order accurate approximation; the second-order accurate parameter-averaging approximation of \cite{wong1993}; the best of the two third-order accurate approximations in \citet{fraser.reid.wong.1997}; and the marginal plausibility interval.  In this case, the marginal plausibility interval is shorter than both the second- and third-order accurate confidence intervals.  
\end{example}

\begin{figure}
\begin{center}
\scalebox{0.55}{\includegraphics{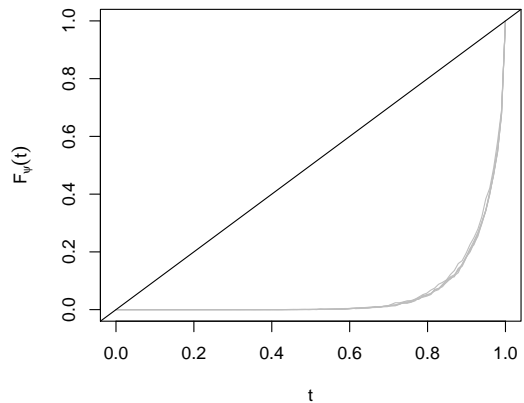}}
\end{center}
\caption{Distribution functions (gray) of the relative profile likelihood $T_{Y,\psi}$ for the gamma mean problem in Example~\ref{ex:gamma.mean}, for mean $\psi=1$ and a range of shapes $\lambda$ from 0.1 to 10.}
\label{fig:gamma.mean}
\end{figure}

\ifthenelse{1=1}{}{
> source("plaus.R")
> psi <- 1
> lambda <- c(0.1, 0.5, 1, 2, 5, 10)
> psi <- c(0.1, 0.5, 1, 2)
> o <- gamma.mean.pl.sim(10, psi, lambda, 1000, 1000); o
      [,1]  [,2]  [,3]  [,4]  [,5]  [,6]
[1,] 0.950 0.947 0.956 0.946 0.950 0.948
[2,] 0.953 0.951 0.950 0.957 0.955 0.950
[3,] 0.951 0.950 0.945 0.931 0.943 0.950
[4,] 0.946 0.954 0.940 0.948 0.945 0.949
> sqrt(0.05 * 0.95 / 1000)
[1] 0.006892024
}

\begin{table}
\begin{center}
\begin{tabular}{cccc}
\hline
& \multicolumn{2}{c}{95\% intervals for $\psi$} & \\
\cline{2-3} 
Method & Lower & Upper & Length \\
\hline
classical & 96.7 & 130.9 & 34.2 \\
\citet{wong1993} & 97.0 & 134.7 & 37.7 \\
\citet{fraser.reid.wong.1997} & 97.2 & 134.2 & 37.0 \\
MPL & 97.1 & 133.6 & 36.5 \\
\hline
\end{tabular}
\end{center}
\caption{Interval estimates for the gamma mean $\psi$ in Example~\ref{ex:gamma.mean}.  First three rows are taken from \citet{fraser.reid.wong.1997}; last row gives the marginal plausibility interval.}  
\label{tab:gamma.mean}
\end{table}

\section{Comparison with parametric bootstrap}
\label{S:bootstrap}

The plausibility function-based method described above allows for the construction of exact frequentist methods in many cases, which is particularly useful in problems where an exact sampling distribution is not available.  An alternative method for such problems is the parametric bootstrap, where the unknown $\prob_\theta$ is replaced by the estimate $\prob_{\hat\theta}$, and the sampling distribution is approximated by simulating from $\prob_{\hat\theta}$.  This approach, and variations thereof, have been carefully studied \citep[e.g.,][]{diciccio.martin.stern.2001, lee.young.2005} and have many desirable properties.  The proposed plausibility function method is, at least superficially, quite similar to the parametric bootstrap, so it is interesting to see how the two methods compare.  In many cases, plausibility functions and parametric bootstrap give similar answers, such as the bivariate normal correlation example above.  Here I show one simple example where the former clearly outperforms the latter.

\begin{example}
\label{ex:random.effects}
Consider a simple Gaussian random effects model, i.e., $Y_1,\ldots,Y_n$ are independently distributed, with $Y_i \sim \nm(\mu_i, \sigma_i^2)$, $i=1,\ldots,n$, where the means $\mu_1,\ldots,\mu_n$ are unknown, but the variances $\sigma_1^2,\ldots,\sigma_n^2$ are known.  The Gaussian random effects portion comes from the assumption that the individual means are an independent $\nm(\lambda,\psi^2)$ sample, where $\theta=(\psi,\lambda)$ is unknown.  Here $\psi \geq 0$ is the parameter of interest, and the overall mean $\lambda$ is a nuisance parameter.  

Using well known properties of Gaussian convolutions, it is possible to recast this hierarchical model in a non-hierarchical form.  That is, $Y_1,\ldots,Y_n$ are independent, with $Y_i \sim \nm(\lambda, \sigma_i^2 + \psi^2)$, $i=1,\ldots,n$.  Here the conditional maximum likelihood estimate of $\lambda$, given $\psi$, is the weighted average $\hat\lambda_\psi = \sum_{i=1}^n w_i(\psi) Y_i / \sum_{i=1}^n w_i(\psi)$, where $w_i(\psi) = 1/(\sigma_i^2 + \psi^2)$, $i=1,\ldots,n$.  From here it is straightforward to write down the relative profile likelihood $T_{Y,\psi}$ in \eqref{eq:profile.likelihood}.  Moreover, since the model is of the composite transformation form, the distribution of $T_{Y,\psi}$ is free of $\lambda$, so any choice of $\lambda$ (e.g., $\lambda=0$) will suffice in the Monte Carlo step \eqref{eq:marginal.monte.carlo}.  One can then readily compute plausibility intervals for $\psi$.  

For interval estimation of $\psi$, a parametric bootstrap is a natural choice.  But it is known that bootstrap tends to have difficulties when the true parameter is at or near the boundary.  The following simulation study will show that the marginal plausibility interval outperforms the parametric bootstrap in the important case of $\psi$ near the boundary, i.e., $\psi \approx 0$.  Since the two-sided bootstrap interval cannot catch a parameter exactly on the boundary, I shall consider true values of $\psi$ getting closer to the boundary as the sample size increases.  In particular, for various $n$, I shall take the true $\psi = n^{-1/2}$ and compare interval estimates based on coverage probability and mean length.  Here data are simulated independently, according to the model $Y_i \sim \nm(0, \sigma_i^2 + \psi^2)$, $i=1,\ldots,n$, where $\psi$ is as above and the $\sigma_1,\ldots,\sigma_n$ are iid samples from an exponential distribution with mean 2.  Table~\ref{tab:random.effects} shows the results of 1000 replications of this process for four sample sizes.  Observe that the plausibility intervals hit the target coverage probability on the nose for each $n$, while the bootstrap suffers from drastic under-coverage for all $n$. 
\end{example}

\begin{table}
\begin{center}
\begin{tabular}{ccccccc}
\hline
& & \multicolumn{2}{c}{MPL} & & \multicolumn{2}{c}{PB} \\
\cline{3-4} \cline{6-7} 
$n$ & & Coverage & Length & & Coverage & Length \\
\hline
50 & & 0.952 & 0.267 & & 0.758 & 0.183 \\
100 & & 0.946 & 0.162 & & 0.767 & 0.138 \\
250 & & 0.948 & 0.079 & & 0.795 & 0.079 \\
500 & & 0.950 & 0.041 & & 0.874 & 0.039 \\
\hline
\end{tabular}
\end{center}
\caption{Estimated coverage probabilities and expected lengths of the 95\% interval estimates for $\psi$ in Example~\ref{ex:random.effects}}  
\label{tab:random.effects}
\end{table}

\ifthenelse{1=1}{}{
\subsubsection{Variable selection in Gaussian linear regression}
\label{SSS:regression}

This final example demonstrates the proposed method's ability to handle non-standard inference problems, in this case, variable selection in regression.  Here I give a relatively simple and straightforward plausibility function-based approach.%; a detailed look at this important problem is beyond the scope of the present paper.  

Consider the Gaussian linear model $Y = \beta_0 1_n + X\beta + \sigma \eps$, where $Y=(Y_1,\ldots,Y_n)^\top$ is a $n$-vector of response variables, $X$ is a (fixed) $n \times p$ matrix of predictor variables, $\beta_0 1_n$ is a constant $n$-vector intercept, $\beta = (\beta_1,\ldots,\beta_p)^\top$ is a $p$-vector of unknown regression coefficients, $\sigma > 0$ is an unknown scale parameter, and $\eps=(\eps_1,\ldots,\eps_n)^\top$ is a $n$-vector of standard Gaussian noise.  The goal is to select a subset of variables (columns of $X$) that suitably describe the variation in the response $Y$.  Towards this, let $\gamma \in \{0,1\}^p$ index the collection of submodels, with $\beta_\gamma$ the corresponding subvector of $\beta$.  For any given $\gamma$, evaluation of the marginal plausibility function for $\beta_\gamma$ is immediate, since the relative profile likelihood for $\beta_\gamma$ has distribution free of all other parameters; in fact, it is distributed as a simple transformation of a $F$-distribution, so no Monte Carlo methods are required.  Now consider the assertion $A_\gamma = \{\beta_{1-\gamma} = 0\}$, i.e., that model $\gamma$ is not wrong.  Evidence in the observed $y$ for model $\gamma$ can, therefore, be measured by 
\[ \mpl_y(A_\gamma) = F_{[\beta_{1-\gamma} = 0]}\bigl( T_{y, [\beta_{1-\gamma}=0]} \bigr), \]
with the subscript ``$[\beta_{1-\gamma}=0]$'' emphasizing the fact that the calculation is based on the marginal plausibility function for $\beta_{1-\gamma}$, evaluated at zero.  As $A_\gamma$ corresponds to a singleton assertion with respect to $\beta_{1-\gamma}$, the frequentist calibration in Theorem~\ref{thm:singleton} holds.  For inference on the number of non-zero coefficients in $\beta$, one can consider 
\begin{equation}
\label{eq:plaus.reg}
\mpl_y(k) = \max_{\gamma: |\gamma| \leq k} \mpl_y(A_\gamma), 
\end{equation}
the plausibility that the model has at most $k$ non-zero coefficients.  It follows from the calibration property quoted above that a rule that rejects the claim of at most $k$ non-zero coefficients when $\mpl_y(k) \leq \alpha$ with control the frequentist Type~I error rate at $\alpha$.  Therefore, a reasonable model selection strategy is to choose $k^\star$ variables, where $k^\star$ is the smallest $k$ such that $\mpl_y(k) \leq \alpha$.  

For a numerical illustration, consider the diabetes data set analyzed in \citet{efron.lars.2004}.  These data consist of observations for $n=442$ diabetes patients with $p=10$ covariates: {\tt age}, {\tt sex}, body mass index ({\tt bmi}), average blood pressure ({\tt map}), and six blood serum measurements ({\tt tc}, {\tt ldl}, {\tt hdl}, {\tt tch}, {\tt ltg}, and {\tt glu}).   In this case, there are $2^{10} = 1024$ possible models; that $\mpl_y(A_\gamma)$ can be evaluated in closed-form, without Monte Carlo, makes the computations very fast, even for much larger $p$.  A plot of $\mpl_y(k)$ versus $k$ for these data is shown in Figure~\ref{fig:diabetes}.  Here, with $\alpha=0.05$, the method selects $k^\star=5$ variables.  Of all models $\gamma$ with five variables, the one with largest $\mpl_y(A_\gamma)$ consists of {\tt sex}, {\tt bmi}, {\tt map}, {\tt hdl}, and {\tt ltg}.  For comparison, these five variables are a subset of the seven selected by lasso \citep{efron.lars.2004} and a superset of the four  selected by the Bayesian lasso \citep{park.casella.2008}.  

\begin{figure}
\begin{center}
\scalebox{0.75}{\includegraphics{diabetes_reg}}
\end{center}
\caption{Plot of the marginal plausibility function $\mpl_y(k)$ in \eqref{eq:plaus.reg} versus $k$ for the diabetes data regression problem in Section~\ref{SSS:regression}.}
\label{fig:diabetes}
\end{figure}
}

\section{Remarks}
\label{S:remarks}

\begin{remark}
\label{re:other.choices}
It was pointed out in Section~\ref{SS:construction} that the relative likelihood \eqref{eq:relative.likelihood} is not the only possible choice for $T_{y,\theta}$.  For example, if the likelihood is unbounded, then one might consider $T_{y,\theta} = L_y(\theta)$.  A penalized version of the likelihood might also be appropriate in some cases, i.e., $T_{y,\theta} = L_y(\theta) \pi(\theta,y)$, where $\pi$ is something like a Bayesian prior (although could depend on $y$ too).  This could be potentially useful in high-dimensional problems.  Another interesting class of $T_{y,\theta}$ quantities are those motivated by higher-order asymptotics, as in \citet{reid2003} and the references therein.  Choosing $T_{y,\theta}$ to be Barndorff-Nielsen's $r^\star = r^\star(\theta,y)$ quantity, or some variation thereof, could potentially give better results, particularly in the marginal inference problem involving $\theta = (\psi,\lambda)$.  However, the possible gain in efficiency comes at the cost of additional analytical computations, and, based on my empirical results, it is unclear if these refinements would lead to any noticeable improvements.  Also, recently, composite likelihoods \citep[e.g.,][]{varin.reid.firth.2011} have been considered in problems where a genuine likelihood is either not available or is too complicated to compute.  The method proposed herein seems like a promising alternative to the bootstrap methods used there, but further investigation is needed. 
\end{remark}

\begin{remark}
\label{re:ds.theory}
There has been considerable efforts to construct a framework of prior-free probabilistic inference; these include fiducial inference \citep{fisher1973}, generalized fiducial inference \citep{hannig2009, hannig2012}, and the Dempster--Shafer theory of belief functions \citep{dempster2008, shafer1976}.  Although $\pl_y$ is not a probability measure, it can be given a prior-free posterior probabilistic interpretation via random sets.  Moreover, the frequentist results presented herein imply that $\pl_y(A)$, as a measure of evidence in support of the claim ``$\theta \in A$,'' is properly calibrated and, therefore, also meaningful across users and/or experiments.  See \citet{imbasics} and \citet{randset} for more along these lines.  In fact, the method presented herein is a sort of generalized version of the inferential model framework developed in \citet{imbasics, imcond, immarg}; details of this generalization shall be fleshed out elsewhere \citep[e.g.][]{imbook}.  
\end{remark}

\begin{remark}
\label{re:monte.carlo}
Using Monte Carlo approximations \eqref{eq:monte.carlo} and \eqref{eq:marginal.monte.carlo} to construct exact frequentist inferential procedures is, to my knowledge, new.  Despite its novelty, the method is surprisingly simple and general.  On the other hand, there is a computational price to pay for this simplicity and generality.  Specifically, determination of plausibility intervals requires evaluation of the Monte Carlo estimate of $F_\theta(T_{y,\theta})$ for several $\theta$ values.  This can be potentially time-consuming, but running the Monte Carlo simulations for different $\theta$ in parallel can help reduce this cost.  The proposed method works---in theory and in principle---in high-dimensional problems, but there the computational cost is further exaggerated.  An important question is if some special techniques can be developed for problems where only the nuisance parameter is high- or infinite-dimensional.  Clever marginalization can reduce the dimension to something manageable within the proposed framework, making exact inference in semiparametric problems possible.  
\end{remark}

\section*{Acknowledgements}

I am grateful for valuable comments from Chuanhai Liu, the Editor, and the anonymous Associate Editor and referees.  This research is partially supported by the National Science Foundation, grant DMS--1208833.

\ifthenelse{1=1}{}{

\appendix

\section{Plausibility function asymptotics}
\label{S:asymptotics}

The results here cover only the case where $T_{Y,\theta}$ is the relative likelihood in \eqref{eq:relative.likelihood}.  However, I expect that similar results would hold for other choices of $T_{Y,\theta}$, although special techniques (maximal inequalities, etc) would likely be needed. 

The goal here is to study the properties of the plausibility function as $n \to \infty$.  Write $\pl_n(\cdot)$ instead of $\pl_Y(\cdot)$ to emphasize the dependence on $n$ through $Y = (Y_1,\ldots,Y_n)$.  Similarly, $L_n$ denotes the likelihood $L_Y$ and $F_{n,\theta}$ the distribution function $F_\theta$ in \eqref{eq:cdf}.  The goal here is to establish consistency of the plausibility function, similar to the consistency of Bayesian posterior distributions \citep[][Sec.~1.3]{ghoshramamoorthi}.  Here I shall say that the plausibility function is \emph{consistent at $\theta^\star$} if $\pl_n(A) \to 0$ with $\prob_{\theta^\star}$-probability~1 for any subset $A$ such that $\theta^\star$ resides outside the closure of $A$.  Such a property implies that the chance of making an incorrect decision with the plausibility function-based procedures in Section~\ref{SS:inference} is asymptotically small.

Consider an assertion $A$ whose closure does not contain $\theta^\star$.  It is shown in \eqref{eq:plausibility.bound1} that, for any $\eta > 0$,  
\begin{equation}
\label{eq:plausibility.bound0}
\prob_{\theta^\star}\{\pl_n(A) > \eta\} \leq \prob_{\theta^\star}\Bigl\{ \sup_{\theta \in A} \frac{L_n(\theta)}{L_n(\theta^\star)} > \inf_{\theta \in A} F_{n,\theta}^{-1}(\eta) \Bigr\}.
\end{equation}
From \eqref{eq:plausibility.bound0} it is clear that proving consistency requires two things: controlling tail probabilities of $\sup_{\theta \in A} T_{Y,\theta}$ and keeping $F_{n,\theta}^{-1}(\eta)$ sufficiently far away from zero for fixed $\eta$.  The first hurdle can be overcome using empirical process theory \citep[e.g.,][]{wongshen1995}.  The second hurdle can be handled with some additional assumptions.

Suppose $(Y_1,\ldots,Y_n) \in \YY^n$ are iid $\prob_\theta$, and that $\prob_\theta$ admits a density $p_\theta$ with respect to a common $\sigma$-finite measure $\nu$ on $\YY$; write $\mathscr{P} = \{p_\theta: \theta \in \Theta\}$.  For any two non-negative $\nu$-integrable functions $(f,g)$ on $\YY$, let $H(f,g) = \int (f^{1/2}-g^{1/2})^2 \,d\nu$ denote the Hellinger distance between $f$ and $g$.  For $p_\theta,p_{\theta'} \in \mathscr{P}$, write $h(\theta,\theta') = H(p_\theta, p_{\theta'})$; if $\theta \mapsto p_\theta$ is one-to-one, then $h(\theta,\theta')$ defines a metric on $\Theta$.  Assume that, at least in a small neighborhood of the true $\theta^\star$, there is a constant $C$ such that $h(\theta^\star,\theta) \leq C \|\theta-\theta^\star\|$, where $\|\cdot\|$ is the natural metric on $\Theta$.  A Hellinger $\eps$-bracket is a set $\{p \in \mathscr{P}: \ell \leq p \leq u\}$, where $(\ell,u)$ are two non-negative $\nu$-integrable functions with $H(\ell,u) < \eps$.  For a subset $\mathscr{P}_0 \subseteq \mathscr{P}$, define $\nbrack(\eps, \mathscr{P}_0, H)$ to be the smallest number of Hellinger $\eps$-brackets needed to cover $\mathscr{P}_0$.  Finally, let $\lesssim$ and $\gtrsim$ denote inequality up to a universal constant.   

\begin{theorem}
\label{thm:consistency}
Let $\theta^\star$ denote the true parameter value, and take any $\eta \in (0,1)$.  In addition to the assumptions in the preceding paragraph, suppose there exists positive constants $a > 0$, $b = b(\eta) > 0$ such that, for any $\eps > 0$ and for all large $n$, 
\begin{equation}
\label{eq:bracket}
\int_{\eps^2/2^8}^{\sqrt{2}\eps} \bigl\{\log\nbrack(x/a, \mathscr{P}_0, H) \bigr\}^{1/2} \,dx \lesssim \sqrt{n} \eps^2, 
\end{equation}
where $\mathscr{P}_0 = \mathscr{P} \cap \{p_\theta: H(p_{\theta^\star},p_\theta) \leq 2\eps\}$, and
\begin{equation}
\label{eq:quantile.bound}
\inf_{\theta: h(\theta^\star, \theta) > \eps} F_{n,\theta}^{-1}(\eta) \geq e^{-b n \eps^2}.   
\end{equation}
Then the plausibility function $\pl_n(\cdot)$ is consistent at $\theta^\star$.    
\end{theorem}

\begin{proof}
\ifthenelse{1=0}{Omitted to save space.}{
First, take any $\eta \in (0,1)$ and any $A \subset \Theta$.  Then 
\begin{align}
\prob_{\theta^\star}\{\pl_n(A) > \eta\} & = \prob_{\theta^\star} \Bigl\{ \sup_{\theta \in A} F_\theta(T_{Y,\theta}) > \eta \Bigr\} \notag \\
& \leq \prob_{\theta^\star}\Bigl\{ \sup_{\theta \in A} F_\theta\Bigl(\sup_{\theta \in A} T_{Y,\theta} \Bigr) > \eta \Bigr\} \notag \\
& = \prob_{\theta^\star} \Bigl\{ \sup_{\theta \in A} T_{Y,\theta} > \inf_{\theta \in A} F_\theta^{-1}(\eta) \Bigr\} \notag \\
& \leq \prob_{\theta^\star}\Bigl\{ \sup_{\theta \in A} \frac{L_n(\theta)}{L_n(\theta^\star)} > \inf_{\theta \in A} F_{n,\theta}^{-1}(\eta) \Bigr\}, \label{eq:plausibility.bound1}
\end{align}
where the last inequality follows because $L_n(\theta^\star) \leq L_n(\hat\theta)$.  This proves \eqref{eq:plausibility.bound0}.  Outer probabilities may be needed on the right-hand side if measurability is not guaranteed.  

Now take any assertion $A \not\ni \theta^\star$.  Then there exists an open $\|\cdot\|$-ball $B$ centered at $\theta^\star$ such that $A \subset B^c$.  Since the Hellinger distance $h$ is bounded, up to a constant, by $\|\cdot\|$, there exists $\eps > 0$ such that $A_\eps = \{\theta: h(\theta^\star,\theta) > \eps\} \supset A$.  From \eqref{eq:plausibility.bound1} and \eqref{eq:quantile.bound}, 
\begin{align}
\prob_{\theta^\star}\{ \pl_n(A) > \eta \} & \leq \prob_{\theta^\star} \Bigl\{ \sup_{\theta \in A} \frac{L_n(\theta)}{L_n(\theta^\star)} > \inf_{\theta \in A} F_{n,\theta}^{-1}(\eta) \Bigr\} \notag \\
& \leq \prob_{\theta^\star}\Bigl\{ \sup_{\theta \in A_\eps} \frac{L_n(\theta)}{L_n(\theta^\star)} > \inf_{\theta \in A_\eps} F_{n,\theta}^{-1}(\eta) \Bigr\} \notag \\
& \leq \prob_{\theta^\star}\Bigl\{ \sup_{\theta \in A_\eps} \frac{L_n(\theta)}{L_n(\theta^\star)} > e^{-bn\eps^2} \Bigr\}.  \label{eq:plausibility.bound2}
\end{align}
By assumption \eqref{eq:bracket} and Theorem~1 of \citet{wongshen1995}, there is a constant $c > 0$ such that the upper bound itself is bounded by $4e^{-cn\eps^2}$.  It follows that, for any $\eta \in (0,1)$, the upper bound \eqref{eq:plausibility.bound2} for $\prob_{\theta^\star}\{\pl_n(A) > \eta\}$ is summable over $n \geq 1$.  An application of the Borel--Cantelli lemma gives that $\pl_n(A) \to 0$ with $\prob_{\theta^\star}$-probability~1.  Since $A \not\ni \theta^\star$ is arbitrary, the plausibility function is consistent.
}  
\end{proof}

Theorem~\ref{thm:consistency} implies that, as $n \to \infty$, the Type~II error probability of the test vanishes and, similarly, the plausibility region is shrinking to the singleton $\{\theta^\star\}$.  In fact, Theorem~\ref{thm:consistency} actually gives more than just consistency---there is some notion of a rate at which $\pl_n(\cdot)$ collapses to a point mass at $\theta^\star$.  For example, if $(\eps_n)$ is a vanishing sequence such that the conditions \eqref{eq:bracket} and \eqref{eq:quantile.bound} hold with $\eps_n$ in place of $\eps$, then the same proof shows that $\pl_n(\{\theta: \|\theta-\theta^\star\| \gtrsim \eps_n\}) \to 0$ in $\prob_{\theta^\star}$-probability, provided that $n \eps_n^2 \to \infty$.  The latter condition can be relaxed to $n \eps_n^2 = O(1)$ with some extra effort.  Compare this to the Bayesian results in, e.g., \citet{ggv2000}.  Therefore, the plausibility regions shrink with $n$ at a rate roughly $n^{-1/2}$.  

Conditions \eqref{eq:bracket} and \eqref{eq:quantile.bound} are rather abstract, so some comments are in order.  First, for parametric families $\mathscr{P}$, the Hellinger bracketing metric entropy $\log\nbrack(x/a, \mathscr{P} \cap \{p_\theta: H(p_{\theta^\star},p_\theta) \leq 2\eps\}, H)$ is typically of the order $\log(a\eps/x)$.  So condition \eqref{eq:bracket} holds with $\eps$ of the order $n^{-1/2}$.  See Remark~(ii) in \citet[][p.~351]{wongshen1995}.  Second, unlike uniform control of likelihood ratios, which is ubiquitous in statistical theory, bounding the quantile function $F_{n,\theta}^{-1}$ seems unique to the present approach.  In group-transformation models, Theorem~\ref{thm:efficiency} says $F_{n,\theta}$ is free of $\theta$ and positive on $(0,1)$; thus, \eqref{eq:quantile.bound} holds trivially.  Outside group-transformation models, if $-2\log T_{Y,\theta}$ is asymptotically chi-square, then the fact that $F_{n,\theta}$ is converging pointwise to a smooth distribution function supported on $(0,1)$ provides some comfort, though \eqref{eq:quantile.bound} should hold more generally.  My conjecture is that $T_{Y,\theta}$ is stochastically larger than $\unif(0,1)$, uniformly in $\theta$, in which case, \eqref{eq:quantile.bound} holds trivially.  This can be checked numerically for many examples (see Figure~\ref{fig:stochastic.order}), but a general proof of this claim escapes me.

\begin{figure}
\begin{center}
\subfigure[Poisson]{\scalebox{0.60}{\includegraphics{poisson_cdf}}} 
\subfigure[Lindley model (Example~\ref{ex:lindley})]{\scalebox{0.60}{\includegraphics{lindley_cdf}}}
\end{center}
\caption{Plots of distribution functions $F_\theta$ for a range of $\theta$ values for two distributions with $n=25$ (in gray).  Diagonal line is the uniform distribution function.  }
\label{fig:stochastic.order}
\end{figure}

}

\bibliographystyle{apalike}
\bibliography{/Users/rgmartin/Dropbox/Research/mybib}

\end{document}